\documentclass{amsart}
\usepackage{amsmath,amsthm,amssymb,amscd}
 \newtheorem{thm}{Theorem}[section]
 
 \newtheorem{lem}[thm]{Lemma}

 \theoremstyle{definition}
 \newtheorem{defn}[thm]{Definition}
 \theoremstyle{remark}

 \numberwithin{equation}{section}

\begin{document}

%
%
%
%
%
%
%
%
%

\title[On Andrews' integer partitions with even parts below odd parts]
 {On Andrews' integer partitions with even parts below odd parts}

\author{Chiranjit Ray}

\address{Harish-Chandra Research Institute, Prayagraj (Allahabad), India, PIN-211019}
\email{chiranjitray.m@gmail.com}

\author{Rupam Barman}
\address{Department of Mathematics, Indian Institute of Technology Guwahati, Assam, India, PIN- 781039}
\email{rupam@iitg.ac.in}

\date{Revised: January 3, 2020}

\thanks{We thank Professor Scott Ahlgren who helped us in finding a proof of an important step in Theorem 1.2. We also thank the referee for helpful comments.
The first author acknowledges the financial support of Department of Atomic Energy, Government of India and Harish-Chandra Research Institute for providing research facilities.}



\dedicatory{}

\begin{abstract}
Recently, Andrews defined a partition function $\mathcal{EO}(n)$ which counts the number of partitions of $n$ in which every even part is less than each odd part. He also defined a partition function 
$\overline{\mathcal{EO}}(n)$ which counts the number of partitions of $n$ enumerated by $\mathcal{EO}(n)$ in which only the largest even part appears an odd number of times. 
Andrews proposed to undertake a more extensive investigation of the properties of $\overline{\mathcal{EO}}(n)$. In this article, we prove infinite families of congruences for $\overline{\mathcal{EO}}(n)$. 
We next study distribution of $\overline{\mathcal{EO}}(n)$. We prove that there are infinitely many integers $N$ in every arithmetic progression for which $\overline{\mathcal{EO}}(2N)$ is even; 
and that there are infinitely many integers $M$ in every arithmetic progression for which $\overline{\mathcal{EO}}(2M)$ is odd so long as there is at least one. We further prove that 
$\overline{\mathcal{EO}}(n)$ is even for almost all $n$. Very recently,  Uncu has treated a different subset of the partitions enumerated by $\mathcal{EO}(n)$. 
We prove that Uncu's partition function is divisible by $2^k$ for almost all $k$. We use arithmetic properties of modular forms and Hecke eigenforms to prove our results.
\end{abstract}

\maketitle
\section{Introduction and statement of results}
A partition of a nonnegative integer $n$ is a nonincreasing sequence of positive integers whose sum is $n$. 
In a recent paper, Andrews \cite{andrews_2018} studied the partition function $\mathcal{EO}(n)$ which counts the 
number of partitions of $n$ where every even part is less than each odd part. He denoted by $\overline{\mathcal{EO}}(n)$, the number of 
partitions counted by $\mathcal{EO}(n)$ in which \textit{only} the largest even part appears an odd number of times. For example, $\mathcal{EO}(8)=12$ with the relevant partitions being 
$8, 6+2, 7+1, 4+4, 4+2+2, 5+3, 5+1+1+1, 2+2+2+2, 3+3+2, 3+3+1+1, 3+1+1+1+1+1, 1+1+1+1+1+1+1+1+1$; and $\overline{\mathcal{EO}}(8)=5$, with the relevant partitions being $8, 4+2+2, 3+3+2, 3+3+1+1, 1+1+1+1+1+1+1+1$.
\par 
Andrews proved that the partition function $\overline{\mathcal{EO}}(n)$ has the following generating function \cite[Eqn. (3.2)]{andrews_2018}:
	\begin{align}\label{gen-2}
		\sum_{n=0}^{\infty}\overline{\mathcal{EO}}(n)q^n=\frac{(q^4; q^4)_{\infty}}{(q^2; q^4)_{\infty}^2}=\frac{(q^4; q^4)_{\infty}^3}{(q^2; q^2)_{\infty}^2},
	\end{align}
where $(a; q)_{\infty}:=\prod_{n\geq 0}(1-aq^n)$. In the same paper, he proposed to undertake a more 
extensive investigation of the properties of $\overline{\mathcal{EO}}(n)$. The objective of this paper is to study divisibility properties of $\overline{\mathcal{EO}}(n)$.
To be specific, we use the theory of Hecke eigenforms to establish the following two infinite families of congruences for $\overline{\mathcal{EO}}(n)$ modulo $2$ and $8$, respectively. 	
\begin{thm}\label{thm1}
	Let $k, n$ be nonnegative integers. For each $i$ with $1\leq i \leq k+1$, if $p_i \geq 5$ is prime such that $p_i \equiv 2 \pmod 3$, then for any integer $j \not\equiv 0 \pmod {p_{k+1}}$
	\begin{align*} 
	\overline{\mathcal{EO}}\left(p_1^2\dots p_{k+1}^2 n + \frac{p_1^2\dots p_{k}^2 p_{k+1}(3j+p_{k+1})-1}{3}\right) \equiv 0 \pmod 2.
\end{align*}
\end{thm}
Let $p\geq 5$ be a prime such that $p\equiv 2\pmod{3}$. By taking all the primes $p_1, p_2, \ldots, p_{k+1}$ to be equal to the same prime $p$ in Theorem \ref{thm1}, 
we obtain the following infinite family of congruences for $\overline{\mathcal{EO}}(n)$:
	\begin{align*} 
	\overline{\mathcal{EO}}\left( p^{2(k+1)}n + p^{2k+1}j + \frac{p^{2(k+1)}-1}{3}\right) \equiv 0 \pmod 2,
	\end{align*}
	where  $j \not\equiv 0 \pmod p$. In particular, for all $n\geq 0$ and $j\not\equiv 0\pmod{5}$, we have
	\begin{align*} 
	\overline{\mathcal{EO}}\left(25n + 5j + 8\right) \equiv 0 \pmod{2}.
	\end{align*}
\begin{thm}\label{thm2}
	Let $k, n$ be nonnegative integers. For each $i$ with $1\leq i \leq k+1$, if $p_i \equiv 1 \pmod{24}$ is prime such that $\overline{\mathcal{EO}}\left(\frac{19p_i-1}{3}\right) \equiv 0 \pmod{8}$, 
	then for any integer $j \not\equiv 0 \pmod{p_{k+1}}$
	\begin{align*}	
	\overline{\mathcal{EO}}\left(8p_1^2\dots p_{k+1}^2 n + \frac{p_1^2\dots p_{k}^2 p_{k+1}(24j+19p_{k+1})-1}{3}\right) \equiv 0 \pmod 8.
	\end{align*}
\end{thm}
Let $p$ be a prime such that $p \equiv 1 \pmod {24}$ and $\overline{\mathcal{EO}}\left(\frac{19p-1}{3}\right) \equiv 0 \pmod{8}$. 
By taking all the primes $p_1, p_2, \ldots, p_{k+1}$ to be equal to the same prime $p$ in Theorem \ref{thm2}, 
we obtain the following infinite family of congruences for $\overline{\mathcal{EO}}(n)$: 
	\begin{align*} 
	\overline{\mathcal{EO}}\left( 8p^{2(k+1)}n + 8p^{2k+1}j + \frac{19p^{2(k+1)}-1}{3}\right) \equiv 0 \pmod 8,
	\end{align*}
	where  $j \not\equiv 0 \pmod{p}$. In particular, if we choose $p = 1009$, then $1009\equiv 1 \pmod{24}$ and $\frac{19\times 1009-1}{3}=6390$. 
	Using \textit{Mathematica} we verify that $\overline{\mathcal{EO}}\left( 6390\right) \equiv 0 \pmod 8$. Thus, for all $n\geq 0$ and $j \not\equiv 0 \pmod{1009}$, we have
	\begin{align*} 
	\overline{\mathcal{EO}}\left(8144648n +  8072 j + 6447846\right) \equiv 0 \pmod{8}.
	\end{align*}
	\par
In \cite{andrews_2018}, Andrews proved that, for all $n\geq 0$ 
\begin{align}\label{andrews-cong}
\overline{\mathcal{EO}}(10n+8) \equiv 0\pmod{5}.
\end{align} 
In this article, we prove that the congruence \eqref{andrews-cong} is also true modulo $4$ if $n\not \equiv 0 \pmod{5}$. To be specific, we prove the following result. 
\begin{thm} \label{thm5} Let $t\in \{1, 2, 3, 4\}$. Then for all $n\geq 0$ we have
	\begin{align*} 
	\overline{\mathcal{EO}}(10(5n+t)+8) \equiv~0\pmod{20}.
	\end{align*}
\end{thm}
We note that Theorem \ref{thm5} is not true if $t=0$. For example, $\overline{\mathcal{EO}}(8)$ is not divisible by $4$.
\par 
For a nonnegative integer $n$, let $p(n)$ denote the number of partitions of $n$. In \cite{ono3}, Ono proved that there are infinitely many integers $N$  in every arithmetic progression for which $p(N)$
is even; and that there are infinitely many integers $M$ in every arithmetic progression for which $p(M)$ is odd so long as there is at least one. Ono's result gave an affirmative answer to a well-known conjecture on 
parity of $p(n)$ in an arithmetic progression. In the following theorem, we prove the same for the partition function $\overline{\mathcal{EO}}(n)$. We note that $\overline{\mathcal{EO}}(2n+1)=0$ 
for all $n\geq 0$.
\begin{thm}\label{thm7}
For any arithmetic progression $r \pmod{t}$, there are infinitely many integers $N\equiv r \pmod {t}$ for which  $\overline{\mathcal{EO}}(2N)$ is even. Also,
for any arithmetic progression $r \pmod {t}$, there are infinitely many integers $M\equiv r \pmod {t}$ for which $\overline{\mathcal{EO}}(2M)$ is odd,  provided there is one such $M$. Furthermore, if there
does exist an $M\equiv r \pmod {t}$ for which $\overline{\mathcal{EO}}(2M)$ is odd, then the smallest such $M$ is less than $$\frac{2^{9+j}3^7t^6}{d^2}\prod_{p\mid 6t}\left(1-\frac{1}{p^2}\right) - 2^j,$$
where $d = \gcd(12r-1,t)$ and $2^j>\displaystyle\frac{t}{12}.$
\end{thm}
A well-known conjecture of Parkin and Shanks \cite{PS} states that the even and odd values of $p(n)$ are equally distributed, that is, 
\begin{align*}
 \lim_{X\to\infty} \frac{\# \left\{0\leq n\leq X: p(n)\equiv r\pmod{2}\right\}}{X}=\frac{1}{2},
\end{align*}
where $r\in \{0, 1\}$. Little is known regarding this conjecture. In the following theorem we prove that $\overline{\mathcal{EO}}(2n)$ is almost always even. 
\begin{thm}\label{thm6} Let $n\geq 0$. Then 
$\overline{\mathcal{EO}}(8n+6)$ is almost always divisible by $8$, namely,
	\begin{align*}
	\lim_{X\to\infty} \frac{\# \left\{0\leq n\leq X: \overline{\mathcal{EO}}(8n+6)\equiv 0\pmod{8}\right\}}{X}=1.
	\end{align*}
\end{thm}
Recently, Uncu \cite{uncu} has treated a different subset of the partitions enumerated by $\mathcal{EO}(n)$. Also see \cite[p. 435]{andrews_2018}. We denote by $\mathcal{EO}_u(n)$ 
the partition function defined by Uncu, and the generating function is given by
\begin{align}\label{gen2}
\sum_{n=0}^{\infty} \mathcal{EO}_u(n)q^n=\frac{1}{(q^2; q^4)_{\infty}^2}.
\end{align}
For any fixed positive integer $k$, Gordon and Ono \cite{ono2} proved that the number of partitions of $n$ into distinct parts is divisible by $2^k$ for almost 
all $n$. Similar studies are done for some other partition functions, for example see \cite{BR, bringman, lin4, Chiranjit-Barman}.  
In this article, we study divisibility of the partition function $\mathcal{EO}_u(n)$ by $2^k$. To be specific, we prove the following result. 
\begin{thm} \label{thm4} Let $k$ be a positive integer. Then $\mathcal{EO}_u(2n)$ is almost always divisible by $2^k$, namely,
	\begin{align*}
	\lim_{X\to\infty} \frac{\# \left\{0\leq n\leq X: \mathcal{EO}_u(2n)\equiv 0\pmod{2^k}\right\}}{X}=1.
	\end{align*}
\end{thm}

\section{Preliminaries}
In this section, we recall some definitions and basic facts on modular forms. For more details, see for example \cite{ono, koblitz}. We first define the matrix groups 
\begin{align*}
\text{SL}_2(\mathbb{Z}) & :=\left\{\begin{bmatrix}
a  &  b \\
c  &  d      
\end{bmatrix}: a, b, c, d \in \mathbb{Z}, ad-bc=1
\right\},\\
\Gamma_{\infty} & :=\left\{
\begin{bmatrix}
1  &  n \\
0  &  1      
\end{bmatrix}: n\in \mathbb{Z}  \right\},\\
\Gamma_{0}(N) & :=\left\{
\begin{bmatrix}
a  &  b \\
c  &  d      
\end{bmatrix} \in \text{SL}_2(\mathbb{Z}) : c\equiv 0\pmod N \right\},\\
\Gamma_{1}(N) & :=\left\{
\begin{bmatrix}
a  &  b \\
c  &  d      
\end{bmatrix} \in \Gamma_0(N) : a\equiv d\equiv 1\pmod N \right\}
\end{align*}
and 
\begin{align*}\Gamma(N) & :=\left\{
\begin{bmatrix}
a  &  b \\
c  &  d      
\end{bmatrix} \in \text{SL}_2(\mathbb{Z}) : a\equiv d\equiv 1\pmod N, ~\text{and}~ b\equiv c\equiv 0\pmod N\right\},
\end{align*}
where $N$ is a positive integer. A subgroup $\Gamma$ of $\text{SL}_2(\mathbb{Z})$ is called a congruence subgroup if $\Gamma(N)\subseteq \Gamma$ for some $N$. The smallest $N$ such that $\Gamma(N)\subseteq \Gamma$
is called the level of $\Gamma$. For example, $\Gamma_0(N)$ and $\Gamma_1(N)$ are congruence subgroups of level $N$. The index of $\Gamma_{0}(N)$ in $\text{SL}_2(\mathbb{Z})$ is
\begin{align*}
 [\text{SL}_2(\mathbb{Z}) : \Gamma_0(N)] = N\prod_{p|N}(1+p^{-1}), 
\end{align*}
where $p$ denotes a prime.
\par Let $\mathbb{H}:=\{z\in \mathbb{C}: \text{Im}(z)>0\}$ be the upper half of the complex plane. The group $$\text{GL}_2^{+}(\mathbb{R})=\left\{\begin{bmatrix}
a  &  b \\
c  &  d      
\end{bmatrix}: a, b, c, d\in \mathbb{R}~\text{and}~ad-bc>0\right\}$$ acts on $\mathbb{H}$ by $\begin{bmatrix}
a  &  b \\
c  &  d      
\end{bmatrix} z=\displaystyle \frac{az+b}{cz+d}$.  
We identify $\infty$ with $\displaystyle\frac{1}{0}$ and define $\begin{bmatrix}
a  &  b \\
c  &  d      
\end{bmatrix} \displaystyle\frac{r}{s}=\displaystyle \frac{ar+bs}{cr+ds}$, where $\displaystyle\frac{r}{s}\in \mathbb{Q}\cup\{\infty\}$.
This gives an action of $\text{GL}_2^{+}(\mathbb{R})$ on the extended upper half-plane $\mathbb{H}^{\ast}=\mathbb{H}\cup\mathbb{Q}\cup\{\infty\}$. 
Suppose that $\Gamma$ is a congruence subgroup of $\text{SL}_2(\mathbb{Z})$. A cusp of $\Gamma$ is an equivalence class in $\mathbb{P}^1=\mathbb{Q}\cup\{\infty\}$ under the action of $\Gamma$.
\par The group $\text{GL}_2^{+}(\mathbb{R})$ also acts on functions $f: \mathbb{H}\rightarrow \mathbb{C}$. In particular, suppose that $\gamma=\begin{bmatrix}
a  &  b \\
c  &  d      
\end{bmatrix}\in \text{GL}_2^{+}(\mathbb{R})$. If $f(z)$ is a meromorphic function on $\mathbb{H}$ and $\ell$ is an integer, then define the slash operator $|_{\ell}$ by 
$$(f|_{\ell}\gamma)(z):=(\text{det}~{\gamma})^{\ell/2}(cz+d)^{-\ell}f(\gamma z).$$
\begin{defn}
Let $\Gamma$ be a congruence subgroup of level $N$. A holomorphic function $f: \mathbb{H}\rightarrow \mathbb{C}$ is called a modular form with integer weight $\ell$ on $\Gamma$ if the following hold:
\begin{enumerate}
 \item We have $$f\left(\displaystyle \frac{az+b}{cz+d}\right)=(cz+d)^{\ell}f(z)$$ for all $z\in \mathbb{H}$ and all $\begin{bmatrix}
a  &  b \\
c  &  d      
\end{bmatrix} \in \Gamma$.
\item If $\gamma\in \text{SL}_2(\mathbb{Z})$, then $(f|_{\ell}\gamma)(z)$ has a Fourier expansion of the form $$(f|_{\ell}\gamma)(z)=\displaystyle\sum_{n\geq 0}a_{\gamma}(n)q_N^n,$$
where $q_N:=e^{2\pi iz/N}$. That is, $f$ is holomorphic at all the cusps of $\Gamma$.
\end{enumerate}
\end{defn}
For a positive integer $\ell$, the complex vector space of modular forms of weight $\ell$ with respect to a congruence subgroup $\Gamma$ is denoted by $M_{\ell}(\Gamma)$. A modular form 
$f\in M_{\ell}(\Gamma)$ is called a cusp form if $f$ vanishes at all the cusps of $\Gamma$. The subspace of $M_{\ell}(\Gamma)$ consisting of cusp forms is denoted by $S_{\ell}(\Gamma)$.
\begin{defn}\cite[Definition 1.15]{ono}
	If $\chi$ is a Dirichlet character modulo $N$, then we say that a modular form $f\in M_{\ell}(\Gamma_1(N))$ has Nebentypus character $\chi$ if
	$$f\left( \frac{az+b}{cz+d}\right)=\chi(d)(cz+d)^{\ell}f(z)$$ for all $z\in \mathbb{H}$ and all $\begin{bmatrix}
	a  &  b \\
	c  &  d      
	\end{bmatrix} \in \Gamma_0(N)$. The space of such modular forms is denoted by $M_{\ell}(\Gamma_0(N), \chi)$. The corresponding space of cusp forms  is denoted by $S_{\ell}(\Gamma_0(N),\chi)$.
	If $\chi$ is the trivial character then we write $M_{\ell}(\Gamma_0(N))$ and $S_{\ell}(\Gamma_0(N))$ for short.
\end{defn}
\par 
Recall that Dedekind's eta-function $\eta(z)$ is defined by
\begin{align*}
	\eta(z):=q^{1/24}(q;q)_{\infty}=q^{1/24}\prod_{n=1}^{\infty}(1-q^n),
\end{align*}
where $q:=e^{2\pi iz}$ and $z\in \mathbb{H}$. A function $f(z)$ is called an eta-quotient if it is of the form
\begin{align*}
f(z)=\prod_{\delta\mid N}\eta(\delta z)^{r_\delta},
\end{align*}
where $N$ is a positive integer and $r_{\delta}$ is an integer. 
\par 
We now recall two theorems from \cite[p. 18]{ono} which will be used to prove our result.
\begin{thm}\cite[Theorem 1.64 and Theorem 1.65]{ono}\label{thm_ono1} If $f(z)=\prod_{\delta\mid N}\eta(\delta z)^{r_\delta}$ 
is an eta-quotient such that $\ell=\frac{1}{2}\sum_{\delta\mid N}r_{\delta}\in \mathbb{Z}$, 
	$$\sum_{\delta\mid N} \delta r_{\delta}\equiv 0 \pmod{24}$$ and
	$$\sum_{\delta\mid N} \frac{N}{\delta}r_{\delta}\equiv 0 \pmod{24},$$
	then $f(z)$ satisfies $$f\left( \frac{az+b}{cz+d}\right)=\chi(d)(cz+d)^{\ell}f(z)$$
	for every  $\begin{bmatrix}
		a  &  b \\
		c  &  d      
	\end{bmatrix} \in \Gamma_0(N)$. Here the character $\chi$ is defined by $\chi(d):=\left(\frac{(-1)^{\ell} \prod_{\delta\mid N}\delta^{r_{\delta}}}{d}\right)$. 
	 In addition, if $c, d,$ and $N$ are positive integers with $d\mid N$ and $\gcd(c, d)=1$, then the order of vanishing of $f(z)$ at the cusp $\frac{c}{d}$ 
	is $\frac{N}{24}\sum_{\delta\mid N}\frac{\gcd(d,\delta)^2r_{\delta}}{\gcd(d,\frac{N}{d})d\delta}$.
\end{thm}
Suppose that $f(z)$ is an eta-quotient satisfying the conditions of Theorem \ref{thm_ono1}. 
If $f(z)$ is holomorphic at all of the cusps of $\Gamma_0(N)$, then $f(z)\in M_{\ell}(\Gamma_0(N), \chi)$. 
\begin{defn}
Let $m$ be a positive integer and $f(z) = \sum_{n=0}^{\infty} a(n)q^n \in M_{\ell}(\Gamma_0(N),\chi)$. Then the action of Hecke operator $T_m$ on $f(z)$ is defined by 
\begin{align*}
f(z)|T_m := \sum_{n=0}^{\infty} \left(\sum_{d\mid \gcd(n,m)}\chi(d)d^{\ell-1}a\left(\frac{nm}{d^2}\right)\right)q^n.
\end{align*}
In particular, if $m=p$ is prime, we have 
\begin{align}\label{hecke1}
f(z)|T_p := \sum_{n=0}^{\infty} \left(a(pn)+\chi(p)p^{\ell-1}a\left(\frac{n}{p}\right)\right)q^n.
\end{align}
We note that $a(n)=0$ unless $n$ is a nonnegative integer.
\end{defn}
\begin{defn}\label{hecke2}
	A modular form $f(z)=\sum_{n=0}^{\infty}a(n)q^n \in M_{\ell}(\Gamma_0(N),\chi)$ is called a Hecke eigenform if for every $m\geq2$ there exists a complex number $\lambda(m)$ for which 
	\begin{align}\label{hecke3}
	f(z)|T_m = \lambda(m)f(z).
	\end{align}
\end{defn}
\section{Proof of Theorems \ref{thm1} and \ref{thm2}}
We use the theory of Hecke eigenforms to prove Theorems \ref{thm1} and \ref{thm2}. 
\begin{proof}[Proof of Theorem \ref{thm1}]
We have 
		\begin{align*}
		\sum_{n=0}^{\infty}\overline{\mathcal{EO}}(n)q^n &=\frac{(q^4; q^4)_{\infty}^3}{(q^2; q^2)_{\infty}^2} \equiv (q; q)_{\infty}^{8} \pmod{2}.
		\end{align*}
This gives
\begin{align*}
\sum_{n=0}^{\infty}\overline{\mathcal{EO}}(n)q^{3n+1} \equiv \eta^8(3z)\pmod 2.
\end{align*}	
Let $\eta^8(3z) = \sum_{n=1}^{\infty} a(n)q^n.$ Then $a(n) = 0$ if $n\not\equiv 1\pmod{3}$ and for all $n\geq 0$, 
\begin{align}\label{1_1}
\overline{\mathcal{EO}}(n) \equiv a(3n+1) \pmod 2.
\end{align}	
By Theorem \ref{thm_ono1}, we have $\eta^8(3z) \in S_4(\Gamma_0(9))$. Since $\eta^8(3z)$ is a Hecke eigenform (see, for example \cite{martin}), \eqref{hecke1} and \eqref{hecke3} yield
\begin{align*}
	\eta^8(3z)|T_p = \sum_{n=1}^{\infty} \left(a(pn) + p^3 a\left(\frac{n}{p}\right) \right)q^n = \lambda(p) \sum_{n=1}^{\infty} a(n)q^n,
\end{align*}
which implies 
\begin{align}\label{1.1}
	a(pn) + p^3 a\left(\frac{n}{p}\right) = \lambda(p)a(n).
\end{align}
Putting $n=1$ and noting that $a(1)=1$, we readily obtain $a(p) = \lambda(p)$.
Since $a(p)=0$ for all $p \not\equiv 1 \pmod{3}$, we have $\lambda(p) = 0$.
From \eqref{1.1}, we obtain 
\begin{align}\label{new-1.1}
a(pn) + p^3 a\left(\frac{n}{p}\right) = 0.
\end{align}
From \eqref{new-1.1}, we derive that for all $n \geq 0$ and $p\nmid r$, 
\begin{align}\label{1.2}
a(p^2n + pr) = 0
\end{align}
and  
\begin{align}\label{1.3}
a(p^2n) = - p^3 a(n)\equiv a(n) \pmod{2}.
\end{align}
Substituting $n$ by $3n-pr+1$ in \eqref{1.2} and together with \eqref{1_1}, we find that
\begin{align}\label{1.4}
\overline{\mathcal{EO}}\left(p^2n + \frac{p^2-1}{3}+ pr\frac{1-p^2}{3}\right) \equiv 0 \pmod 2.
\end{align}	
Substituting $n$ by $3n+1$ in \eqref{1.3} and using \eqref{1_1}, we obtain
\begin{align}\label{1.5}
\overline{\mathcal{EO}}\left(p^2n + \frac{p^2-1}{3}\right) \equiv \overline{\mathcal{EO}}(n) \pmod 2.
\end{align}
Since $p \geq 5$ is prime, so $3\mid (1-p^2)$ and $\gcd \left(\frac{1-p^2}{3} , p\right) = 1$.  
Hence when $r$ runs over a residue system excluding the multiple of $p$, so does $\frac{1-p^2}{3}r$. Thus \eqref{1.4} can be rewritten as
\begin{align}\label{1.6}
\overline{\mathcal{EO}}\left(p^2n + \frac{p^2-1}{3}+ pj\right) \equiv 0 \pmod 2,
\end{align}
where $p \nmid j$. 
\par Now, $p_i \geq 5$ are primes such that $p_i \not\equiv 1 \pmod 3$. Since 
\begin{align*}
p_1^2\dots p_{k}^2n + \frac{p_1^2\dots p_{k}^2-1}{3}=p_1^2\left(p_2^2\dots p_{k}^2n + \frac{p_2^2\dots p_{k}^2-1}{3}\right)+\frac{p_1^2-1}{3},
\end{align*}
using \eqref{1.5} repeatedly we obtain that
\begin{align}\label{1.7}
\overline{\mathcal{EO}}\left(p_1^2\dots p_{k}^2n + \frac{p_1^2\dots p_{k}^2-1}{3}\right) \equiv \overline{\mathcal{EO}}(n) \pmod{2}.
\end{align}
Let $j\not\equiv 0\pmod{p_{k+1}}$. Then \eqref{1.6} and \eqref{1.7} yield
\begin{align*}
\overline{\mathcal{EO}}\left(p_1^2\dots p_{k+1}^2n + \frac{p_1^2\dots p_{k}^2p_{k+1}(3j+p_{k+1})-1}{3}\right) \equiv 0 \pmod{2}.
\end{align*}
This completes the proof of the theorem.
\end{proof}
To prove Theorem \ref{thm2}, we need that the eta-quotient $\eta^5(96z)/\eta(24z)$ is an eigenform for the Hecke operators $T_p$, where $p\equiv 1\pmod{24}$. This has been observed to be true by Scott Ahlgren.
We now present below the proof given by Ahlgren which was communicated to us through an email. 
Let $F_1=\eta^5(24z)/\eta(96z)$, $F_7=\eta^3(24z)\eta(96z)$, $F_{13}=\eta(24z)\eta^3(96z)$, and $F_{19}=\eta^5(96z)/\eta(24z)$. Then $F_j$ is supported on exponents congruent to $j \pmod{24}$. 
The Hecke operators $T_p$ for $p\equiv 5, 11, 17, 23\pmod{24}$ annihilate each of these forms. The Hecke operators $T_p$ for $p\equiv 1, 5, 13, 19 \pmod{24}$ map $F_j$ to a multiple of $F_{j'}$, 
where $j'\equiv pj\pmod{24}$. It turns out that a linear combination of the forms $F_j$ is an eigenform of all of the Hecke operators. 
In \cite[p. 209]{kohler}, equation (13.84) expresses the linear combination as an eigenform.
Since the $F_j$ are supported on distinct classes of coefficients, it follows that $F_j$ are eigenforms of all the Hecke operators. 
\begin{proof}[Proof of Theorem \ref{thm2}]
We first recall the following $2$-dissection formula from \cite[Entry 25, p. 40]{berndt}: 
	\begin{align}\label{2_3}
			\frac{1}{(q;q)_{\infty}^2} = \frac{(q^8;q^8)_{\infty}^5}{(q^2;q^2)_{\infty}^5(q^{16};q^{16})_{\infty}^2} 
			+ 2q \frac{(q^4;q^4)_{\infty}^2(q^{16};q^{16})_{\infty}^2}{(q^2;q^2)_{\infty}^5(q^{8};q^{8})_{\infty}}.
	\end{align}
From \eqref{gen-2}, we have 
\begin{align}\label{new-100}
 \sum_{n=0}^{\infty}\overline{\mathcal{EO}}(2n)q^n=\frac{(q^2; q^2)_{\infty}^3}{(q; q)_{\infty}^2}.
\end{align}
Combining \eqref{2_3} and \eqref{new-100}, and then extracting the terms with odd powers of $q$, we deduce that
\begin{align}\label{new-101}
 \sum_{n=0}^{\infty}\overline{\mathcal{EO}}(4n+2)q^n = 2 \frac{(q^2;q^2)_{\infty}^2(q^8;q^8)_{\infty}^2}{(q;q)_{\infty}^2(q^4; q^4)_{\infty}}.
\end{align}
We again combine \eqref{2_3} and \eqref{new-101}, and then extract the terms with odd powers of $q$ to obtain 
\begin{align*}
\sum_{n=0}^{\infty}\overline{\mathcal{EO}} (8n+6)q^n = 4 \frac{(q^2;q^2)_{\infty}(q^4;q^4)_{\infty}(q^8;q^8)_{\infty}^2}{(q;q)_{\infty}^3}.
\end{align*}	
Since $(q; q)_{\infty}^2 \equiv (q^2;q^2)_{\infty} \pmod{2}$, we have
\begin{align*}
\sum_{n=0}^{\infty}\overline{\mathcal{EO}} (8n+6)q^n \equiv 4 \frac{(q^4;q^4)_{\infty}^5}{(q;q)_{\infty}} \pmod 8.
\end{align*}	
This gives
	\begin{align*}
	\sum_{n=0}^{\infty}\overline{\mathcal{EO}} (8n+6)q^{24n+19} \equiv 4 \frac{\eta(96z)^5}{\eta(24z)} \pmod{8}. 
	\end{align*}
Let  $\frac{\eta(96z)^5}{\eta(24z)}= \sum_{n=1}^{\infty}a(n)q^n$. It is clear that $a(n)=0$ if $n\not\equiv 19 \pmod{24}.$ Also, for all $n\geq 0$,
\begin{align}\label{2}
	\overline{\mathcal{EO}}(8n+6)\equiv 4 a(24n+19) \pmod{8}.
\end{align}
By Theorem \ref{thm_ono1}, we have $\frac{\eta(96z)^5}{\eta(24z)} \in S_2\left(\Gamma_0(2304)\right)$.  
Since $\frac{\eta(96z)^5}{\eta(24z)}$ is a Hecke eigenform for the Hecke operator $T_p$, where $p\equiv 1 \pmod{24}$, \eqref{hecke1} and \eqref{hecke3} yield  
\begin{align}\label{2.1}
a(pn) + p  \left(\frac{2}{p}\right)  a\left(\frac{n}{p}\right) = \lambda(p)a(n).
\end{align}	
Putting $n=19$ in \eqref{2.1} and noting that $p \not\equiv 19 \pmod{24}$, we obtain $a(19p) = \lambda (p) a(19)$.
Also, $a(19)=1$, and hence $a(19p) = \lambda (p)$. Thus \eqref{2.1} gives 
\begin{align}\label{new-102}
a(pn) + p  \left(\frac{2}{p}\right)  a\left(\frac{n}{p}\right) = a(19p)a(n).
\end{align}
From \eqref{new-102}, we obtain that for all $n\geq 0$ and $p\nmid r$,
\begin{align}\label{2.2}
a(p^2n) + a\left(n\right) \equiv a(19p)a(pn)\pmod{2}
\end{align}
and
\begin{align}\label{2.3}
a(p^2n+pr) = a(19p)a(pn+r).
\end{align}
Let $A(n) = a(24n+19)$. Let $p$ be a prime such that $p\equiv 1\pmod{24}$. Now, 
replacing $n$ by $24n-pr+19$ in \eqref{2.3}, we obtain
\begin{align}\label{2.4}
	A\left(p^2n+ 19\frac{p^2-1}{24}+pr\frac{1-p^2}{24}\right)
	= A\left(19\frac{p-1}{24}\right)A\left(pn+ 19\frac{p-1}{24}+r\frac{1-p^2}{24}\right).
\end{align}
We note that $\gcd\left(\frac{1-p^2}{24} , p\right) = 1$.   
Hence when $r$ runs over a residue system excluding the multiple of $p$, so does $\frac{1-p^2}{24}r$.
Thus, \eqref{2.4} can be rewritten as  
\begin{align}\label{2.5}
A\left(p^2n+ 19\frac{p^2-1}{24}+pj\right)
= A\left(19\frac{p-1}{24}\right)A\left(pn+ 19\frac{p-1}{24}+j\right),
\end{align}
where $p\nmid j$.
Similarly, replacing $n$ by $24n+19$ in \eqref{2.2}, we have, modulo $2$
\begin{align}\label{2.6}
A\left(p^2n+ 19\frac{p^2-1}{24}\right) +A(n)
\equiv A\left(19\frac{p-1}{24}\right)A\left(pn+ 19\frac{p-1}{24}\right).
\end{align}	
Let $p$ be such that $\overline{\mathcal{EO}}\left(\frac{19p-1}{3}\right)\equiv 0\pmod{8}$. Then, using the relation $\overline{\mathcal{EO}} (8n+6)\equiv 4A(n)\pmod{8}$, 
we have $A\left(19\frac{p-1}{24}\right)\equiv 0\pmod{2}$. 
Hence, \eqref{2.5} and \eqref{2.6} imply
\begin{align}\label{2.7}
A\left(p^2n+ 19\frac{p^2-1}{24}+pj\right)\equiv 0\pmod{2}
\end{align}
and 
\begin{align}\label{2.8}
A\left(p^2n+ 19\frac{p^2-1}{24}\right) 
\equiv  A(n) \pmod{2}.
\end{align}	
From our hypothesis, we have $p_i \geq 5$ are primes such that $p_i \equiv 1 \pmod{24}$ and $A\left(19\frac{p_i-1}{24}\right) \equiv 0 \pmod{2}$. Now, using \eqref{2.8} we deduce that
\begin{align*}
A\left(p_1^2\dots p_{k}^2n+ 19\frac{p_1^2\dots p_{k}^2-1}{24}\right) 
\equiv  A(n) \pmod 2.
\end{align*}	
Replacing $n$ by $p_{k+1}^2n+ 19\frac{p_{k+1}^2-1}{24} +p_{k+1}j$, and then using \eqref{2.7} we obtain 
\begin{align*}
A\left(p_1^2\dots p_{k}^2p_{k+1}^2n+ 19\frac{p_1^2\dots p_{k}^2p_{k+1}^2-1}{24}+ p_1^2\dots p_{k}^2p_{k+1}j\right) 
\equiv  0 \pmod 2.
\end{align*}
We complete the proof by using the fact that $\overline{\mathcal{EO}} (8n+6)\equiv 4A(n)\pmod{8}$.
\end{proof}

	\section{Proof of Theorem \ref{thm5}}
We prove Theorem \ref{thm5} using the approach developed in \cite{radu1, radu2}. To this end, we first recall some definitions and results 
from \cite{radu1, radu2}. For a positive integer $M$, let $R(M)$ be the set of integer sequences $r=(r_\delta)_{\delta\mid M}$ indexed by the positive divisors of $M$. 
If $r \in R(M)$ and $1=\delta_1<\delta_2< \cdots <\delta_k=M$ 
are the positive divisors of $M$, we write $r=(r_{\delta_1}, \ldots, r_{\delta_k})$. Define $c_r(n)$ by 
\begin{align}
\sum_{n=0}^{\infty}c_r(n)q^n:=\prod_{\delta\mid M}(q^{\delta};q^{\delta})^{r_{\delta}}_{\infty}=\prod_{\delta\mid M}\prod_{n=1}^{\infty}(1-q^{n \delta})^{r_{\delta}}.
\end{align}
The approach to proving congruences for $c_r(n)$ developed by Radu \cite{radu1, radu2} reduces the number of cases that one must check as compared with the classical method which uses Sturm's bound alone.
\par 
Let $m$ be a positive integer. For any integer $s$, let $[s]_m$ denote the residue class of $s$ in $\mathbb{Z}_m:= \mathbb{Z}/ {m\mathbb{Z}}$. 
Let $\mathbb{Z}_m^{*}$ be the set of all invertible elements in $\mathbb{Z}_m$. Let $\mathbb{S}_m\subseteq\mathbb{Z}_m$  be the set of all squares in $\mathbb{Z}_m^{*}$. For $t\in\{0, 1, \ldots, m-1\}$
and $r \in R(M)$, we define a subset $P_{m,r}(t)\subseteq\{0, 1, \ldots, m-1\}$ by
\begin{align*}
P_{m,r}(t):=\left\{t': \exists [s]_{24m}\in \mathbb{S}_{24m} ~ \text{such} ~ \text{that} ~ t'\equiv ts+\frac{s-1}{24}\sum_{\delta\mid M}\delta r_\delta \pmod{m} \right\}.
\end{align*}
\begin{defn}
	Suppose $m, M$ and $N$ are positive integers, $r=(r_{\delta})\in R(M)$ and $t\in \{0, 1, \ldots, m-1\}$. Let $k=k(m):=\gcd(m^2-1,24)$ and write  
	\begin{align*}
	\prod_{\delta\mid M}\delta^{|r_{\delta}|}=2^s\cdot j,
	\end{align*}
	where $s$ and $j$  are nonnegative integers with $j$ odd. The set $\Delta^{*}$ consists of all tuples $(m, M, N, (r_{\delta}), t)$ satisfying these conditions and all of the following.
	\begin{enumerate}
		\item Each prime divisor of $m$ is also a divisor of $N$.
		\item $\delta\mid M$ implies $\delta\mid mN$ for every $\delta\geq1$ such that $r_{\delta} \neq 0$.
		\item $kN\sum_{\delta\mid M}r_{\delta} mN/\delta \equiv 0 \pmod{24}$.
		\item $kN\sum_{\delta\mid M}r_{\delta} \equiv 0 \pmod{8}$.  
		\item  $\frac{24m}{\gcd{(-24kt-k{\sum_{{\delta}\mid M}}{\delta r_{\delta}}},24m)}$ divides $N$.
		\item If $2\mid m$, then either $4\mid kN$ and $8\mid sN$ or $2\mid s$ and $8\mid (1-j)N$.
	\end{enumerate}
\end{defn}
Throughout this section we take $\Gamma=\text{SL}_2(\mathbb{Z})$. Let $m, M, N$ be positive integers. For $\gamma=
\begin{bmatrix}
	a  &  b \\
	c  &  d     
\end{bmatrix} \in \Gamma$, $r\in R(M)$ and $r'\in R(N)$, set 
	\begin{align*}
	p_{m,r}(\gamma):=\min_{\lambda\in\{0, 1, \ldots, m-1\}}\frac{1}{24}\sum_{\delta\mid M}r_{\delta}\frac{\gcd^2(\delta a+ \delta k\lambda c, mc)}{\delta m}
	\end{align*}
and 
	\begin{align*}
	p_{r'}^{*}(\gamma):=\frac{1}{24}\sum_{\delta\mid N}r'_{\delta}\frac{\gcd^2(\delta, c)}{\delta}.
	\end{align*}
	\begin{lem}\label{lem1}\cite[Lemma 4.5]{radu1} Let $u$ be a positive integer, $(m, M, N, r=(r_{\delta}), t)\in\Delta^{*}$ and $r'=(r'_{\delta})\in R(N)$. 
	Let $\{\gamma_1,\gamma_2, \ldots, \gamma_n\}\subseteq \Gamma$ be a complete set of representatives of the double cosets of $\Gamma_{0}(N) \backslash \Gamma/ \Gamma_\infty$. 
	Assume that $p_{m,r}(\gamma_i)+p_{r'}^{*}(\gamma_i) \geq 0$ for all $1 \leq i \leq n$. Let $t_{min}=\min_{t' \in P_{m,r}(t)} t'$ and
	\begin{align*}
	\nu:= \frac{1}{24}\left\{ \left( \sum_{\delta\mid M}r_{\delta}+\sum_{\delta\mid N}r'_{\delta}\right)[\Gamma:\Gamma_{0}(N)] -\sum_{\delta\mid N} \delta r'_{\delta}\right\}-\frac{1}{24m}\sum_{\delta\mid M}\delta r_{\delta} 
	- \frac{ t_{min}}{m}.
 	\end{align*}	
	If the congruence $c_r(mn+t')\equiv0\pmod u$ holds for all $t' \in P_{m,r}(t)$ and $0\leq n\leq \lfloor\nu\rfloor$, then it holds for all $t'\in P_{m,r}(t)$ and $n\geq0$.
	\end{lem}
	To apply Lemma \ref{lem1} we utilize the following result, which gives a complete set of representatives of the double cosets in  
	$\Gamma_{0}(N) \backslash \Gamma/ \Gamma_\infty$. 
	\begin{lem}\label{lem2}\cite[Lemma 4.3]{wang} If $N$ or $\frac{1}{2}N$ is a square-free integer, then
		\begin{align*}
		\bigcup_{\delta\mid N}\Gamma_0(N)\begin{bmatrix}
		1  &  0 \\
		\delta  &  1      
		\end{bmatrix}\Gamma_ {\infty}=\Gamma.
		\end{align*}
	\end{lem}
\begin{proof}[Proof of Theorem \ref{thm5}]
Due to \eqref{andrews-cong} we need to prove our congruences modulo $4$ only.
We have 
\begin{align*}
	 \sum_{n=0}^{\infty}\overline{\mathcal{EO}}(n)q^n=\frac{(q^4; q^4)_{\infty}^3}{(q^2; q^2)_{\infty}^2}&=\frac{(q^2; q^2)_{\infty}^2(q^4; q^4)_{\infty}^3}{(q^2; q^2)_{\infty}^4}\\
	 &\equiv\frac{(q^2; q^2)_{\infty}^2(q^4; q^4)_{\infty}^3}{(q^4; q^4)_{\infty}^2}\pmod{4}\\
	&= (q^2; q^2)_{\infty}^2(q^4; q^4)_{\infty}\pmod{4}.
\end{align*}
Let $(m,M,N,r,t)=(50,8,10,(0,2,1,0),18)$. It is easy to verify that $(m,M,N,r,t) \in \Delta^{*}$ and $P_{m,r}(t)=\{18, 28, 38, 48\}$.
   From  Lemma \ref{lem2} we know that $\left\{\begin{bmatrix}
	1  &  0 \\
	\delta  &  1      
	\end{bmatrix}:\delta|10 \right\}$ forms a complete set of double coset representatives of $\Gamma_{0}(N) \backslash \Gamma/ \Gamma_\infty$.
	Let $r'=(0,0,0,0,0,0)\in R(10)$. We have used $Sage$ to verify that
	$p_{m,r}(\gamma_{\delta})+p_{r'}^{*}(\gamma_{\delta}) \geq 0$ for each $\delta \mid N$, where $\gamma_{\delta}=\begin{bmatrix}
	1  &  0 \\
	\delta  &  1      
	\end{bmatrix}$. We compute that the upper bound in Lemma \ref{lem1} is $\lfloor\nu\rfloor=1$. 
	Using $Mathematica$ we verify that
	$\overline{\mathcal{EO}}(50n+t') \equiv 0 \pmod{4}$ for $n \leq 1$ and $t'\in P_{m, r}(t)$.
	Thus, by Lemma \ref{lem1}, we conclude that  $\overline{\mathcal{EO}}(50n+t') \equiv 0 \pmod{4}$ for any $n\geq0$, where $t'\in \{18, 28, 38, 48\}$.
	This completes the proof of the theorem.
	 \end{proof}
	\section{Proof of Theorems \ref{thm7}, \ref{thm6} and \ref{thm4}}
We prove Theorem \ref{thm7} by using the approach developed in \cite{ono3}. Recently, Jameson and Wieczorek \cite{Jameson} have done a similar study for the generalized Frobenius partitions.
To make this paper self-contained, we recall two results from \cite{Jameson}. Also see \cite{ono3}. Let $M_k^{!}\left(\Gamma_0(N_0), \chi \right)$ denote the space of weakly holomorphic modular forms.
\begin{thm}\label{Theorem-5}\cite[Theorem 5]{Jameson}
 Let $N_0, \alpha, \beta, t$ be integers with $N_0, \alpha, t$ positive, and let 
 \begin{align*}
  \sum_{n=0}^{\infty}c(n)q^{\alpha n+\beta}\in M_k^{!}\left(\Gamma_0(N_0), \chi \right),
 \end{align*}
where $c(n)$ are algebraic integers in some number field. For any arithmetic progression $r \pmod{t}$, there are infinitely many integers $N\equiv r \pmod{t}$ for which $c(N)$ is even.
\end{thm}
\begin{thm}\label{Theorem-6}\cite[Theorem 6]{Jameson}
Let $N_0, \alpha, \beta, t$ be integers with $N_0, \alpha$ positive, and $t>1$, and let 
 \begin{align*}
  \sum_{n=0}^{\infty}c(n)q^{\alpha n+\beta}\in M_k^{!}\left(\Gamma_0(N_0), \chi \right),
 \end{align*}
where $c(n)$ are algebraic integers in some number field. For any arithmetic progression $r \pmod{t}$, there are infinitely many integers $M\equiv r \pmod{t}$ for which $c(M)$ is odd,
provided there is one such $M$.
\par Furthermore, if there does exist an $M\equiv r\pmod{t}$ for which $c(M)$ is odd, then the smallest such $M$ is less than $C_{r, t}$ for
\begin{align*}
 C_{r, t}:=\displaystyle \frac{2^j \cdot 12 + k}{12\alpha}\left[\displaystyle \frac{N\alpha^2 t^2}{d}\right]^2\prod_{p\mid N\alpha t}\left(1-\frac{1}{p^2}\right)-2^j,
\end{align*}
where $N:=\emph{lcm}(\alpha t, N_0)$, $d:=\gcd(\alpha r+\beta, t)$, and $j$ is a sufficiently large integer.
\end{thm}
\begin{proof}[Proof of Theorem \ref{thm7}]
We have
\begin{align*}
	\sum_{n=0}^{\infty}\overline{\mathcal{EO}} (2n)q^n = \frac{(q^2;q^2)_{\infty}^3}{(q;q)_{\infty}^2}.
\end{align*}	
We rewrite the above identity in terms of $\eta$-quotients, and then use the binomial theorem to obtain
\begin{align*}
\sum_{n=0}^{\infty}\overline{ \mathcal{EO}} (2n)q^{6n+1} = \frac{\eta(12z)^3}{\eta(6z)^2} \equiv \frac{\eta(12z)^4}{\eta(6z)^4} \pmod {2}.
\end{align*}
By Theorem \ref{thm_ono1}, we have 
\begin{align*}
\frac{\eta(12z)^4}{\eta(6z)^4} \in M_0^{!}\left(\Gamma_0(72)\right).
\end{align*}
\par Let
$f_t(z) :=\displaystyle \frac{\eta(12z)^4}{\eta(6z)^4}\Delta^{2^j}(6tz)$, where $\Delta(z):=\eta^{24}(z)$. The cusps of $\Gamma_0(72t)$ are represented by fractions $\frac{c}{d}$ where $d\mid 72t$ and 
$\gcd(c, d)=1$. Now, $f_t(z)$ vanishes at the cusp $\frac{c}{d}$ if and only if 
\begin{align*}
	&4\frac{\gcd(d,12)^2}{12} - 4\frac{\gcd(d,6)^2}{6} +24\cdot 2^j\frac{\gcd(d,6t)^2}{6t}>0.
\end{align*}
We have \begin{align*}
	&4\frac{\gcd(d,12)^2}{12} - 4\frac{\gcd(d,6)^2}{6} +24\cdot 2^j\frac{\gcd(d,6t)^2}{6t}\geq 2^j\frac{6}{t} - \frac{1}{2}.
\end{align*}
Hence, if $j$ is an integer such that $2^j>\frac{t}{12}$, then $f_t(z) \in S_{12\cdot 2^{j}}\left(\Gamma_0(72t)\right).$
Finally, our desired result follows immediately by applying Theorems \ref{Theorem-5} and \ref{Theorem-6} to $\sum_{n=0}^{\infty}\overline{ \mathcal{EO}} (2n)q^{6n+1}$. 
\end{proof}
\begin{proof}[Proof of Theorem \ref{thm6}]
We first recall the following $2$-dissection formula from \cite[Entry 25, p. 40]{berndt}: 
	\begin{align}\label{2_3}
			\frac{1}{(q;q)_{\infty}^2} = \frac{(q^8;q^8)_{\infty}^5}{(q^2;q^2)_{\infty}^5(q^{16};q^{16})_{\infty}^2} 
			+ 2q \frac{(q^4;q^4)_{\infty}^2(q^{16};q^{16})_{\infty}^2}{(q^2;q^2)_{\infty}^5(q^{8};q^{8})_{\infty}}.
	\end{align}
From \eqref{gen-2}, we have 
\begin{align}\label{new-100}
 \sum_{n=0}^{\infty}\overline{\mathcal{EO}}(2n)q^n=\frac{(q^2; q^2)_{\infty}^3}{(q; q)_{\infty}^2}.
\end{align}
Combining \eqref{2_3} and \eqref{new-100}, and then extracting the terms with odd powers of $q$, we deduce that
\begin{align}\label{new-101}
 \sum_{n=0}^{\infty}\overline{\mathcal{EO}}(4n+2)q^n = 2 \frac{(q^2;q^2)_{\infty}^2(q^8;q^8)_{\infty}^2}{(q;q)_{\infty}^2(q^4; q^4)_{\infty}}.
\end{align}
We again combine \eqref{2_3} and \eqref{new-101}, and then extract the terms with odd powers of $q$ to obtain 
\begin{align*}
\sum_{n=0}^{\infty}\overline{\mathcal{EO}} (8n+6)q^n = 4 \frac{(q^2;q^2)_{\infty}(q^4;q^4)_{\infty}(q^8;q^8)_{\infty}^2}{(q;q)_{\infty}^3}.
\end{align*}	
Since $(q; q)_{\infty}^2 \equiv (q^2;q^2)_{\infty} \pmod{2}$, we have
\begin{align*}
\sum_{n=0}^{\infty}\overline{\mathcal{EO}} (8n+6)q^n \equiv 4 \frac{(q^4; q^4)_{\infty}^5}{(q;q)_{\infty}} \pmod{8}.
\end{align*}	
We rewrite the above equation in terms of $\eta$-quotients and obtain
\begin{align}\label{new-120}
\sum_{n=0}^{\infty}\overline{\mathcal{EO}} (8n+6)q^{24n+19} \equiv 4 \frac{\eta^5(96z)}{\eta(24z)} \pmod{8}.
\end{align}
Let $A(z)=\displaystyle\frac{\eta^2(24z)}{\eta(48z)}$. Then, $A^2(z)\equiv 1\pmod{4}$.
Also, let $B(z)=\displaystyle \frac{\eta^5(96z)\eta^3(24z)}{\eta^2(48z)}$. Then we have 
\begin{align}\label{new-121}
B(z)=\displaystyle \frac{\eta^5(96z)}{\eta(24z)}A^2(z)\equiv \displaystyle \frac{\eta^5(96z)}{\eta(24z)} \pmod{4}.
\end{align}
The cusps of $\Gamma_0(2304)$ are represented by fractions $\frac{c}{d}$ where $d\mid 2304$ and 
$\gcd(c, d)=1$. By Theorem \ref{thm_ono1}, $B(z)$ is holomorphic at the cusp $\frac{c}{d}$ if and only if
\begin{align*}
 5\frac{\gcd(d,96)^2}{96} + 3\frac{\gcd(d,24)^2}{24} -2\frac{\gcd(d, 48)^2}{48}\geq 0.
\end{align*}
Now,
\begin{align*}
&5\frac{\gcd(d,96)^2}{96} + 3\frac{\gcd(d,24)^2}{24} -2\frac{\gcd(d, 48)^2}{48}\\
&=\frac{\gcd(d,48)^2}{24}\left(\frac{5}{4}\frac{\gcd(d,96)^2}{\gcd(d,48)^2} + 3\frac{\gcd(d,24)^2}{\gcd(d,48)^2}-1\right) \\
& > 0.
\end{align*}
Hence, by Theorem \ref{thm_ono1}, $B(z) \in S_{3}(\Gamma_{0}(2304), \left(\frac{-4}{\bullet}\right))$.
\par
Let $m$ be a positive integer. By a deep theorem of Serre \cite[p. 43]{ono}, if $f(z)\in M_{\ell}(\Gamma_0(N), \chi)$ 
has Fourier expansion 
	$$f(z)=\sum_{n=0}^{\infty}c(n)q^n\in \mathbb{Z}[[q]],$$
	then there is a constant $\alpha>0$  such that
	$$ \# \left\{n\leq X: c(n)\not\equiv 0 \pmod{m} \right\}= \mathcal{O}\left(\frac{X}{(\log{}X)^{\alpha}}\right).$$
Since $B(z) \in S_{3}(\Gamma_{0}(2304), \left(\frac{-4}{\bullet}\right))$, the Fourier coefficients of $B(z)$ are almost always divisible by $m$. Hence, using \eqref{new-121} and \eqref{new-120} 
we complete the proof of the theorem.
\end{proof}	
	\begin{proof}[Proof of Theorem \ref{thm4}]
The generating function of $\mathcal{EO}_u(2n)$ is given by 
	\begin{align}\label{4}
	\sum_{n=0}^{\infty}\mathcal{EO}_u(2n)q^n=\frac{1}{(q; q^2)_{\infty}^2}=\frac{(q^2; q^2)_{\infty}^2}{(q; q)_{\infty}^2}.
	\end{align}
We note that  $\eta(24z)= q\prod_{n=1}^{\infty}(1-q^{24n})$ is a power series of $q$. As in the proof of Theorem \ref{thm6}, let 
	\begin{align*}
	A(z) = \prod_{n=1}^{\infty} \frac{(1-q^{24n})^2}{(1-q^{48n})} = \frac{\eta^2(24z)}{\eta(48z)}. 
	\end{align*}
Then using binomial theorem we have 
		\begin{align}\label{4.1}
		A^{2^k}(z) = \frac{\eta^{2^{k+1}}(24z)}{\eta^{2^k}(48z)} \equiv 1 \pmod {2^{k+1}}.
		\end{align}
Define $B_k(z)$ by
	\begin{align}\label{4.2}
	B_k(z) = \left(\frac{\eta(48z)}{\eta(24z)}\right)^2A^{2^k}(z).
	\end{align}
Modulo $2^{k+1}$, we have
		\begin{align}\label{new-110}
		B_k(z) = \frac{\eta^2(48z)}{\eta^2(24z)}A^{2^k}(z)\equiv \frac{\eta^2(48z)}{\eta^2(24z)} = q^2  \frac{(q^{48}; q^{48})_{\infty}^2}{(q^{24}; q^{24})_{\infty}^2}.
		\end{align}
Combining \eqref{4} and \eqref{new-110}, we obtain  
\begin{align}\label{4.3}
B_k(z) \equiv \sum_{n=0}^{\infty}\mathcal{EO}_u(2n)q^{24n+2} \pmod {2^{k+1}}.
\end{align}
The cusps of $\Gamma_0(576)$ are represented by fractions $\frac{c}{d}$ where $d\mid 576$ and 
$\gcd(c, d)=1$. By Theorem \ref{thm_ono1}, it is easily seen that $B_k(z)$ is a form of weight $2^{k-1}$ on $\Gamma_{0}(576)$. Therefore,  $B_k(z) \in M_{2^{k-1}}(\Gamma_{0}(576))$
if and only if $B_k(z)$ is holomorphic at the cusp $\frac{c}{d}$. We know that $B_k(z)$ is holomorphic at a cusp $\frac{c}{d}$ if and only if
\begin{align*}
 \frac{\gcd(d,24)^2}{24} \left(2^{k+1}-2\right) + \frac{\gcd(d,48)^2}{24}\left(1-2^{k-1}\right) \geq 0.
\end{align*}
Now,
\begin{align*}
&\gcd(d,24)^2 \left(2^{k+1}-2\right) + \gcd(d,48)^2\left(1-2^{k-1}\right)\\
&=\gcd(d,48)^2\left(\frac{\gcd(d,24)^2}{\gcd(d,48)^2} (2^{k+1}-2) + (1 - 2^{k-1})\right) \\
&\geq \frac{1}{4} (2^{k+1}-2) + (1 - 2^{k-1})\\
& > 0.
\end{align*}
Hence, $B_k(z) \in M_{2^{k-1}}(\Gamma_{0}(576))$. Now, using Serre's theorem \cite[p. 43]{ono} as shown in the proof of Theorem \ref{thm6}, we arrive at the desired result due to \eqref{4.3}.
\end{proof}

\end{document}